\documentclass{amsart}
\usepackage{amsmath,amssymb}
\usepackage{amsthm}
\usepackage{bm}
\usepackage{ascmac}
\usepackage[T1]{fontenc}

\theoremstyle{definition}
\newtheorem{dfn}{Definition}[section]
\newtheorem{thm}[dfn]{Theorem}
\newtheorem{prop}[dfn]{Proposition}
\newtheorem{lem}[dfn]{Lemma}

\newtheorem{cor}[dfn]{Corollary}
\newtheorem{rem}[dfn]{Remark}

\renewcommand{\labelenumi}{(\theenumi)}
\numberwithin{equation}{section}
\begin{document}
\title{On the Gromov--Hausdorff stability of metric viscosity solutions}
\author{Shimpei Makida} \thanks{Department of Mathematics, Graduate School of Science, Hokkaido University, North 10, West 8, Kita-Ku, Sapporo 060-0810, JAPAN}
\begin{abstract}
We establish the stability of metric viscosity solutions to first–order Hamilton--Jacobi equations under Gromov--Hausdorff convergence.
Our proof combines a characterization of metric viscosity solutions via quadratic distance functions with a doubling variable method adapted to $\epsilon$-isometries, which allows us to pass to the Gromov--Hausdorff limit without embedding the spaces into a common ambient space. 
As a byproduct, we give a PDE-based proof of the stability of the dual Kantorovich problems under measured-Gromov--Hausdorff convergence.
\end{abstract}
\maketitle
\section{Introduction}
Hamilton--Jacobi equations play a pivotal role in interfacial science and optimal control. From the perspective of PDE, the well-posedness and qualitative properties of their viscosity solutions have been studied intensively since the 1970s. In recent decades, this line of inquiry has been extended beyond Euclidean spaces to a wide range of non-Euclidean and metric settings \cite{AF15,CL85,GHN15,GS14,GS15,LSZ25}.
Notably, applications of Hamilton–Jacobi equations on metric spaces have attracted considerable attention; examples include the construction of viscosity solutions via graph approximations of fractals \cite{CCM16}, domain-perturbation problems \cite{M23,MN23}, and an application to a large-deviation principle \cite{GTT22}.

Our main aim is to study the Gromov--Hausdorff stability of metric viscosity solutions, which is a natural generalization of the aforementioned stability results for domain perturbations \cite{M23,MN23}. 
Parallel to these developments within PDE, such stability has also been extensively studied in broader mathematical contexts, motivated in particular by advances in metric geometry.
In \cite{GMS15}, Gigli, Mondino, and Savaré studied the stability of heat flows under pointed-measured-Gromov--Hausdorff convergence. 
In \cite{H2018}, Honda proved the stability of solutions to elliptic PDEs under Gromov--Hausdorff convergence.
The present work furnishes a first-order analogue, establishing Gromov--Hausdorff stability for metric viscosity solutions of Hamilton–Jacobi equations.

Let $n \in \mathbb{N}$. Suppose that $(X_n, d_n)$ and $(X_\infty, d_\infty)$ are geodesic spaces. We assume
$$
X_n \to X_{\infty}\quad
\text{in Gromov--Hausdorff sense}.
$$
We recall that $X_{n}$ converges to $X_{\infty}$ in Gromov--Hausdorff sense if there exists an $\epsilon_{n}$-isometry $f_{n}:X_{n} \to X_{\infty}$ with $\epsilon_{n} \to 0$ as $n \to \infty$. See Definition \ref{e-iso}.

We consider the Hamilton--Jacobi equation on $X_{n}$
\begin{align}
    \partial_{t}u^n(t,x)+H_n\bigl(x,|\nabla u^n(t,x)|\bigr)&=0,\quad (t,x) \in (0,\infty) \times X_{n}, \label{e:ghj1}\\
    u^n(0,x)&=g^n(x),\quad x \in X_{n}, \label{e:ghj2}
\end{align}
where the Hamiltonian $H_n: X_n \times [0,\infty) \to \mathbb{R}$ is continuous and the initial function $g^n:X_n \to \mathbb{R}$ is continuous.
We also consider the Hamilton--Jacobi equation on $X_{\infty}$
\begin{align}
\partial_t u^\infty(t,x) 
  + H_\infty\bigl(x,|\nabla u^\infty(t,x)|\bigr)
  &= 0, \quad (t,x) \in (0,\infty)\times X_\infty,
  \label{e:limhj1}\\
u^\infty(0,x)
  &= g^\infty(x), \quad x \in X_\infty,
  \label{e:limhj2}
\end{align}
where the Hamiltonian $H_\infty: X_\infty \times [0,\infty) \to \mathbb{R}$ is continuous and the initial function $g^\infty:X_\infty \to \mathbb{R}$ is continuous.

For the stability, we assume the following conditions.
To handle metric viscosity solutions, we extend the domain of Hamiltonian $H_n$ and $H_\infty$ to $X_n \times \mathbb{R}$ and $X_\infty \times \mathbb{R}$, respectively, by constant extension. Precisely, for
\(x \in X_n\), \(y \in X_\infty\), and for \(p<0\), we set
$
    H_n(x,p)      := H_n(x,0), 
    \quad
    H_\infty(y,p) := H_\infty(y,0).
$
\begin{enumerate}
\renewcommand{\labelenumi}{(S1)}
\renewcommand{\theenumi}{(S1)}
\item 
\label{i:S1} ({Gromov--Hausdorff convergence})
    $X_{n} \to X_{\infty}$ in Gromov--Hausdorff sense. In other words,
there exists an $\epsilon_n$-isometry $f_n: X_n \to X_\infty$ with $\epsilon_n \to 0$ as $n \to \infty$.
\renewcommand{\labelenumi}{(S2)}
\renewcommand{\theenumi}{(S2)}
\item 
\label{i:S2} ({Monotonicity of $H_n$ in the second variable})
    Let $x \in X_n$.
    For $p, q \in \mathbb{R}$ with $p\le q$, the inequality
    $$
    H_n(x,p)\le H_n(x,q)
    $$
    holds.
\renewcommand{\labelenumi}{(S3)}
\renewcommand{\theenumi}{(S3)}
\item 
\label{i:S3} ({Uniform convergence of $H_n$}) For $x \in X_\infty$, $p \in \mathbb{R}$, and $x_{n} \in X_n$, $p_n \in \mathbb{R}$ such that $f_n(x_n) \to x$ in $X_\infty$ and $p_{n} \to p$ in $\mathbb{R}$ as $n \to \infty$, the convergence
    $$
    H_n(x_n,p_n) \to H_\infty(x,p)
    $$
    holds.
\renewcommand{\labelenumi}{(S4)}
\renewcommand{\theenumi}{(S4)}
\item 
\label{i:S4} ({Uniform boundedness of $u^n$}) For $T_1, T_2 \in (0,\infty)$ with $T_1 < T_2$, there exists a positive constant $M$ independent of $n$ such that, for $s \in [T_1, T_2]$ and $x \in X_n$, the inequality
    $$
    |u^n(s,x)|\le M
    $$ holds.
\renewcommand{\labelenumi}{(S5)}
\renewcommand{\theenumi}{(S5)}
\item 
\label{i:S5} ({Equi-continuity and pointwise convergence of $g^n$})
There exists a modulus of continuity $\eta$ such that, for $x,y \in X_n$, the inequality 
    $$
    |g^n(x)-g^n(y)|\le \eta\bigl(d_n(x,y)\bigr)
    $$ holds.
Furthermore, for $x \in X_\infty$, $g^n\bigl(f^{'}_{n}(x)\bigr)$ converges to $g^\infty(x)$ as $n \to \infty$, where $f^{'}_{n}: X_{\infty} \to X_n$ is an $\epsilon_n$-inverse of $f_n$. See Definition \ref{e-inv}.
\end{enumerate}

Under these assumptions, we can prove the Gromov--Hausdorff stability of metric viscosity solutions. 
\begin{thm}[Stability under Gromov--Hausdorff convergence]\label{stability1}
Let $(X_n, d_n)$ and $(X_\infty, d_\infty)$ be locally compact geodesic spaces; that is, geodesic spaces in which every closed ball is compact.
Let $u^{n}$ be a metric viscosity subsolution (resp. supersolution) of (\ref{e:ghj1}), (\ref{e:ghj2}).
Suppose that \ref{i:S1}--\ref{i:S5} hold and, for $x \in X_\infty$, the inequality (continuity at the initial time of $u^n$) holds:
$$
    \sup_{(t_n,x_n)}\limsup_{n \to \infty}(u^n)^{\ast}(t_n,x_n)\le\sup_{(x_n)}\limsup_{n \to \infty}(u^n)^{\ast}(0,x_n) 
    $$ 
    $$
    \left(\text{resp.}\; \inf_{(t_n,x_n)}\liminf_{n \to \infty}(u^n)_{\ast}(t_n,x_n)\ge\inf_{(x_n)}\liminf_{n \to \infty}(u^n)_{\ast}(0,x_n)\right)
    $$
    where $t_n \in [0,\infty)$ satisfies $t_n \to 0$ in $[0,\infty)$ and $x_n \in X_n$ satisfies $f_n(x_n) \to x$ in $X_\infty$ as $n \to \infty$.
    Here, $(u^n)^{\ast}$ and $(u^n)_{\ast}$ denote the upper and lower semicontinuous envelopes of $u^n$, respectively.
Define the limit function $u^{\infty}(t, x)$ $(\text{resp.}\; u_{\infty}(t, x))$ by
$$
 u^{\infty}(t, x):=
 \sup_{(t_n,x_n)} \limsup_{n \to \infty} (u^{n})^{\ast}(t_n, x_{n}), \quad (t,x) \in [0,\infty) \times X_{\infty}.
$$
$$
\left(\text{resp.}\;u_{\infty}(t, x):=
\inf_{(t_n, x_n)} \liminf_{n \to \infty} (u^{n})_{\ast}(t_n, x_{n}), \quad (t,x) \in [0,\infty) \times X_{\infty}.
\right)
$$
We also assume that $u^{\infty}$ $(\text{resp.}\; u_{\infty})$ is an upper (resp. lower) semicontinuous function on $X_\infty$.
Then, the function $u^{\infty}$ $(\text{resp.}\; u_{\infty})$ 
is a metric viscosity subsolution (resp. supersolution) of (\ref{e:limhj1}), (\ref{e:limhj2}).
In the definition of $u^{\infty}$ ($\text{resp.}\;u_{\infty}$), the supremum (resp. infimum) is taken over all sequences $t_n \in [0,\infty)$ and $x_n \in X_n$ such that $t_n \to t$ in $[0,\infty)$ and $f_n(x_n) \to x$ in $X_\infty$ as $n \to \infty$.
\end{thm}
Note that if we assume \ref{i:S1}--\ref{i:S4} and the upper (resp. lower) semicontinuity of $u^\infty$ (resp. $u_\infty$), we can prove the stability of metric viscosity subsolutions (resp. supersolutions) to (\ref{e:ghj1}).

We outline the strategy of the proof. First, by using the characterization of metric viscosity solutions via quadratic distance functions \cite{MN23}, we show that $u^\infty$ is a metric viscosity subsolution of (\ref{e:limhj1}).
See Proposition \ref{equaivalence1}.
Next, we demonstrate that $u^\infty$ satisfies the initial condition (\ref{e:limhj2}) using \ref{i:S5}.
A similar argument shows that $u_\infty$ is a metric viscosity supersolution of (\ref{e:limhj1}), (\ref{e:limhj2}).
Our strategy is a natural adaptation of the approach used to prove the stability under Hausdorff convergence \cite{M23,MN23}.

As an application of the stability of metric viscosity solutions, we give a PDE-based proof of the convergence of maximizers of the dual Kantorovich problems under measured-Gromov--Hausdorff convergence.
The dual Kantorovich problem includes the Hopf--Lax semigroup, which can be identified as the unique metric viscosity solution of (\ref{e:ghj1}), (\ref{e:ghj2}) with a specific Hamiltonian. Therefore, we can apply our stability results to the stability of maximizers.

Let $\mu_\infty, \nu_\infty \in P_2(X_\infty)$ and $\mu_n, \nu_n \in P_2(X_n)$.
Here, $P_2(X_\infty)$ denotes the set of probability measures on \( X_\infty \) with finite second moments.
We consider the dual Kantorovich problem on $X_n$
\begin{equation}\tag{{KAN}$_n$}\label{kann} \hspace{2pt}
\max_{\phi \in L^{1}(\mu_{n})}\left(\int_{X_n}Q_{1} \phi \,d\nu_n
-\int_{X_n} \phi \,d\mu_n \right)
\end{equation}
and the dual Kantorovich problem on $X_\infty$
\begin{equation}\tag{KAN$_{\infty}$}\label{kani} 
\max_{\phi \in L^{1}(\mu_{\infty})}\left(\int_{X_{\infty}}Q_{1} \phi \,d\nu_{\infty}
-\int_{X_{\infty}} \phi \,d\mu_{\infty} \right),
\end{equation}
where $Q_1 \phi$ denotes the Hopf--Lax semigroup applied to $\phi$ at time $1$.
See Subsection 2.3 for a more detailed explanation of the Kantorovich problem.
\begin{prop}
[{%
  \cite[Exercise 28.12]{V09}; see also the claim in the proof of
  Theorem 1.1 in \cite{V08}%
}]
\label{stakan}
Let $(X_n, d_n)$ and $(X_\infty, d_\infty)$ be compact geodesic spaces.
Assume \ref{i:S1}.
In other words, there exists an $\epsilon_n$-isometry $f_n$ with $\epsilon_n \to 0$ as $n \to \infty$.
Suppose further that $f_n$ is a Borel map and the pushforwards of measures $(f_n)_{\#}\mu_n$, $(f_n)_{\#}\nu_n$ satisfy
$$
(f_n)_{\#}\mu_n \to \mu_{\infty},\quad(f_n)_{\#}\nu_n \to \nu_{\infty}
$$ in $2$-Wasserstein distance $W_2$ as $n \to \infty$.
Let $\phi^n$ be a maximizer of (\ref{kann}).
Then, there exists a constant $c^n \in \mathbb{R}$ such that a maximizer $\phi^n-c^n$ of (\ref{kann}) converges uniformly to some maximizer $\overline{\phi}$ of (\ref{kani}) up to a subsequence. Equivalently, for $x \in X_{\infty}$, we have
$$
\lim_{n \to \infty} \left(\phi^{n}(x_{n})-c^n\right)= \overline{\phi}(x),
$$
where $x_n \in X_n$ satisfies $f_n(x_n) \to x$ in $X_\infty$ as $n \to \infty$.
\end{prop}

The proof is carried out according to the following strategy. 
First, we show that, for each \( n \), there exists a real number \( c^n \) such that $\phi^n - c^n$
is equi-Lipschitz in $n$ and uniformly bounded. 
Then, by applying the Ascoli--Arzelà theorem, we deduce the existence of a uniformly convergent limit \( \overline{\phi} \). 
Up to this step, the approach is the same as that used in previous study \cite{V08}. 
However, instead of applying the Ascoli--Arzelà theorem directly to \(Q_1(\phi^n - c^n)\) as in the previous work, we deduce the uniform convergence of \(Q_1(\phi^n - c^n)\) from the stability of metric viscosity solutions.  
Using the uniform convergence of these functions, we verify that the limit function \(\overline{\phi}\) is a maximizer of the dual Kantorovich problem on \(X_\infty\).

The remainder of this paper is organized as follows. 
In Section 2, we give a short review of metric viscosity solutions, Gromov--Hausdorff convergence, and the dual Kantorovich problem. 
In Section 3, we prove the stability of metric viscosity solutions under Gromov--Hausdorff convergence.
In Section 4, we prove the stability of maximizers of the dual Kantorovich problems under Gromov--Hausdorff convergence.
\section{Preliminaries}
In this section, we summarize the minimal knowledge of metric viscosity solutions, Gromov--Hausdorff convergence, and the dual Kantorovich problem for the reader's convenience. 
\subsection{Metric viscosity solutions}
Let $(X,d)$ be a metric space satisfying the following condition. For $x,y \in X$, there exists a continuous curve $\gamma: [0,1] \to X$ such that $\gamma(0)=x$, $\gamma(1)=y$ and
$$
d\bigl(\gamma(s),\gamma(t)\bigr)=|s-t|d(x,y),\quad s,t \in [0,1].
$$
A metric space with this property is called a geodesic space.

In this subsection, we consider the Hamilton--Jacobi equation with a continuous Hamiltonian $H(x,p): X\times [0,\infty) \to \mathbb{R}$
\begin{align}
    \partial_{t}u(t,x)+H(x,|\nabla u(t,x)|)&=0,\quad (t,x) \in (0,\infty) \times X, \label{e:hj1}\\
    u(0,x)&=g(x),\quad x \in X,\label{e:hj2}
\end{align}
where $g:X \to \mathbb{R}$ is continuous.

We define the local slope for a locally Lipschitz function $\phi: (0,\infty) \times X \to \mathbb{R}$ by 
$$
|\nabla \phi(t,x)|:=\limsup_{y \to x}\frac{|\phi(t,y)-\phi(t,x)|}{d(y,x)}.
$$
In the same way, we define the lower and upper local slope for a locally Lipschitz function $\phi: (0,\infty)\times X \to \mathbb{R}$ by
\begin{align*}
|\nabla^{-}\phi(t,x)|&:=\limsup_{y \to x}\frac{\max\{0,-\phi(t,y)+\phi(t,x)\}}{d(y,x)},\\ 
|\nabla^{+}\phi(t,x)|&:=\limsup_{y \to x}\frac{\max\{0,\phi(t,y)-\phi(t,x)\}}{d(y,x)}.
\end{align*}
\begin{dfn}[{\cite[Definition 2.2]{GS15}}]
Let $\phi_1, \phi_2:(0,\infty)\times X \to \mathbb{R}$ be locally Lipschitz functions. We call $(\phi_1,\phi_2)$ a subsolution (resp. supersolution) test function if $(\phi_1,\phi_2)$ satisfies the following conditions.
\begin{enumerate}
    \item The identity $|\nabla \phi_{1}(t,x)|=|\nabla^{-}\phi_{1}(t,x)|$ $\left(\text{resp.}\;|\nabla \phi_{1}(t,x)|=|\nabla^{+}\phi_{1}(t,x)|\right)$ holds on $(0,\infty) \times X$ and the slope $|\nabla\phi_{1}(t,x)|$ is continuous on $(0,\infty) \times X$.
    \item The time derivatives $\partial_{t}\phi_{1}$ and $\partial_{t}\phi_{2}$ are continuous on $(0,\infty) \times X$.
\end{enumerate}
We denote by $\underline{C}$ ($\text{resp.}\;\overline{C}$) the collection of all subsolution (resp. supersolution) test functions.
\end{dfn}
For example, by \cite[Lemma 7.2]{GS14}, the function $(\phi_1, \phi_2)=\left(d(a,\cdot)^2,0\right)$, $a \in X$ is a subsolution test function. 

To define the notion of metric viscosity solutions, we extend the domain of the Hamiltonian to $X \times \mathbb{R}$ by constant extension.
For $f: [0,\infty) \times X \to \mathbb{R}$, we define the upper (resp. lower) semicontinuous envelope by
$$
f^{*}(t,x):=\limsup_{(s,y)\to (t,x)}f(s,y).
$$
$$
\left(\text{resp.}\;f_{*}(t,x):=\liminf_{(s,y)\to (t,x)}f(s,y).\right)
$$
\begin{dfn}[{\cite[Definition 2.5]{GS15}}]
Let $u: [0,\infty) \times X \to \mathbb{R}$ be a locally bounded function. 
We say that $u$ is a metric viscosity subsolution (resp. supersolution) of (\ref{e:hj1}) if $u^{*}-\phi_1-\phi_2$, $(\phi_{1},\phi_{2}) \in \underline{C}$ ($\text{resp.}\;u_{*}-\phi_1-\phi_2$, $(\phi_1,\phi_2) \in \overline{C}$) attains a local maximum (resp. minimum) at $(t,x)$ in $(0,\infty) \times X$, then the inequality
$$
\partial_{t} \phi_1(t,x)+\partial_t \phi_2(t,x)+H_{|\nabla \phi_{2}(t,x)|^{*}}\bigl(x,|\nabla \phi_{1}(t,x)|\bigr)\le 0
$$
$$
\left(\text{resp.}\;\partial_{t} \phi_1(t,x)+\partial_t \phi_2(t,x)+H^{|\nabla \phi_{2}(t,x)|^{*}}\bigl(x,|\nabla \phi_{1}(t,x)|\bigr)\ge 0\right)
$$
holds.
Here, for $\eta\ge 0$, we define
$$
H_{\eta}(x,p):=\inf_{|s-p|\le \eta}H(x,s).
$$
$$
\left(\text{resp.}\;H^{\eta}(x,p):=\sup_{|s-p|\le \eta}H(x,s).\right)
$$
We say that $u$ is a metric viscosity subsolution (resp. supersolution) of (\ref{e:hj1}), (\ref{e:hj2}) if $u$ is a metric viscosity subsolution (resp. supersolution) of (\ref{e:hj1}) and $u$ satisfies the initial condition (\ref{e:hj2}):
$$
u(0,x) \le g(x). \quad \Bigl(\text{resp.}\; u(0,x) \ge g(x).\Bigr)
$$
We say that $u$ is a metric viscosity solution of (\ref{e:hj1}), (\ref{e:hj2}) if $u$ is a metric viscosity subsolution and supersolution of (\ref{e:hj1}), (\ref{e:hj2}).
\end{dfn}
\subsection{Gromov--Hausdorff convergence}
In this paper, we use the definition of Gromov--Hausdorff convergence via an $\epsilon$-isometry.
It is known that there are some characterizations of Gromov--Hausdorff convergence. See \cite{V09} for more details.
 \begin{dfn}\label{e-iso}
 Let $(X,d_X)$, $(Y, d_Y)$ be metric spaces.
 Let $\epsilon>0$. We call a map $f: X \to Y$ an $\epsilon$-isometry if $f$ satisfies the following conditions.
 \begin{enumerate}
     \item 
     For $x,y \in X$, the inequality
     $$
    \left|d_Y\bigl(f(x),f(y)\bigr)-d_X(x,y)\right|\le \epsilon
     $$
     holds.
     \item
     The inequality
     $
     d_H\bigl(f(X),Y\bigr) \le \epsilon
     $
     holds, where $d_H$ denotes the Hausdorff distance in $Y$.
 \end{enumerate}
 \end{dfn}
\begin{dfn}
Let $(X_n, d_n)$ and $(X_\infty, d_\infty)$ be metric spaces. We say that $X_n$ converges to $X_\infty$ in Gromov--Hausdorff sense if there exists an $\epsilon_n$-isometry $f_n: X_n \to X_\infty$ with $\epsilon_n \to 0$ as $n \to \infty$.
\end{dfn}
For an $\epsilon$-isometry $f$, one can construct a map $f^{'}$ that serves as an approximate inverse to $f$. 
\begin{dfn}\label{e-inv}
Let $(X,d_X)$, $(Y, d_Y)$ be metric spaces. Let $f:X \to Y$ be an $\epsilon$-isometry. We call a map $f^{'}:Y \to X$ an $\epsilon$-inverse of $f$ if $f^{'}$ satisfies the following conditions.    
\begin{enumerate}
\item The map $f^{'}$ is a $4\epsilon$-isometry.
\item For $x \in X$ and $y \in Y$, the inequalities 
$$
d_X\bigl(f^{'}\circ f(x),x\bigr)\le 3\epsilon, \quad d_Y\bigl(f\circ f^{'}(y),y\bigr)\le \epsilon
$$
hold.
\end{enumerate}
\end{dfn}
\begin{rem}\label{points}
Let $(X_n, d_n)$ and $(X_\infty, d_\infty)$ be metric spaces.
Suppose that $X_n$ converges to $X_\infty$ in Gromov--Hausdorff sense. Then, for $x \in X_\infty$, there exists $x_n \in X_n$ such that $f_n(x_n) \to x$ as $n \to \infty$. 
Indeed, there exists $x_n \in X_n$ such that $d_\infty\bigl(x,f_n(x_n)\bigr) \le d_\infty\bigl(x,f_n(X_n)\bigr)+\frac{1}{n}$. 
Recalling the fact that
$$
d_H\bigl(f_n(X_n),X_\infty\bigr):=\max\left\{\sup_{y\in f_n(X_n)}d_\infty(y,X_\infty), \sup_{y\in X_\infty}d_\infty\bigl(y,f_n(X_n)\bigr)\right\},
$$
we have
$$
d_\infty\bigl(x,f_n(x_n)\bigr) \le d_\infty\bigl(x,f_n(X_n)\bigr)+\frac{1}{n}\le d_H\bigl(f_n(X_n),X_\infty\bigr)+\frac{1}{n}\le \epsilon_n+\frac{1}{n}.
$$
Letting $n \to \infty$, we obtain the desired conclusion.
\end{rem}
\subsection{The dual Kantorovich problem}

Let $(X,d)$ be a compact geodesic space. 
Let $\mu, \nu$ be probability measures on $X$ and let $\Pi(\mu,\nu)$ be the set of probability measures $\pi$ on $X \times X$ satisfying $(p_1)_{\#}\pi=\mu$ and $(p_2)_{\#}\pi=\nu$. 
Here, $p_1, p_2: X\times X \to X$ are projection maps onto the first and second factors, respectively.
For $\mu, \nu$, we define $2$-Wasserstein distance $W_2(\mu,\nu)$ by
$$
W_2(\mu,\nu)=\left(\inf_{\pi \in \Pi(\mu,\nu)} \int_{X \times X} d(x,y)^2 d\pi(x,y)\right)^\frac{1}{2}.
$$
It is well known that \( W_2 \) satisfies the axioms of a distance on \( P_2(X) \) \cite{V09}.

Furthermore, $W_2$ admits the following dual formulation \cite[Theorem 5.10]{V09}:
\begin{equation*}\tag{KAN}\label{kan}
\frac{1}{2}W_2(\mu,\nu)^2=
\max_{\phi \in L^{1}(\mu)}\left(\int_{X} \phi^c \,d\nu
-\int_{X} \phi \,d\mu \right).
\end{equation*}
The right hand side of (\ref{kan}) is called the dual Kantorovich problem.

We now turn to the admissible class of functions $\phi$ in the dual formulation and the definition of their $\frac{d^2}{2}$-transform $\phi^c$.
Since \eqref{kan} admits a maximizer that is \(\frac{d^{2}}{2}\)-convex, we may, without loss of generality, restrict the admissible class of functions \(\phi\) in \eqref{kan} to those that are \(\frac{1}{2} d^{2}\)-convex.
Recall that $\phi:X \to \mathbb{R} \cup \{\pm \infty\}$ is a $\frac{d^2}{2}$-convex function if there exists a function $\zeta:X \to \mathbb{R} \cup \{\pm \infty\}$ such that 
$$
\phi(x)=\sup_{y \in X}\left(\zeta(y)-\frac{1}{2}d(x,y)^2\right),\quad x \in X.
$$
Then, for such a $\phi$, the $\frac{d^2}{2}$-transform $\phi^c$ is defined by 
$$
\phi^c(x)=\inf_{y \in X}\left( \phi(y)+\frac{1}{2}d(x,y)^2\right),\quad x \in X.
$$
Thus, if we use the Hopf--Lax semigroup 
$$
Q_t \phi(x):=\inf_{y \in X}\left( \phi(y)+\frac{1}{2t}d(x,y)^2\right),
$$
the dual Kantorovich problem can be written as
$$
\max_{\phi \in L^{1}(\mu)}\left(\int_{X} Q_1 \phi \,d\nu
-\int_{X} \phi \,d\mu \right).
$$
\section{Stability of metric viscosity solutions of Hamilton--Jacobi equations}
We use a characterization of metric viscosity solutions via quadratic distance functions.
\begin{prop}[{\cite[Proposition 7.2]{MN23}}] 
\label{equaivalence1}
We assume that $(X,d)$ is a locally compact geodesic space.
Then, the following statements are equivalent.
\begin{enumerate}
\item
A function $u: (0, \infty)\times X \to \mathbb{R}$ is a metric viscosity subsolution (resp. supersolution) of \eqref{e:hj1}.
\item
Let $\hat{a}, \hat{x} \in X$, $k \ge 0$, and $\phi$ be a $C^1((0,\infty))$ function.
If $u^*(t, x)-\phi(t)-\frac{k}{2}d(\hat{a}, x)^2$ ($\text{resp.}\;u_*(t, x)-\phi(t)+\frac{k}{2}d(\hat{a}, x)^2$) attains a strict local maximum (resp. minimum) at $(\hat{t}, \hat{x}) \in (0,\infty) \times X$,
then the inequality
$$
\partial_t \phi(\hat{t})+H\bigl(\hat{x}, k d(\hat{a}, \hat{x})\bigr) \le 0
$$
$$
\left(\text{resp.}\;\partial_t \phi(\hat{t})+H\bigl(\hat{x}, k d(\hat{a}, \hat{x})\bigr) \ge 0\right)
$$
holds.
\end{enumerate}
\end{prop}
We are now prepared to present the proof of Theorem \ref{stability1}.
\begin{proof}[Proof of  Theorem \ref{stability1}]
We first prove that $u^{\infty}$ is a metric viscosity subsolution of (\ref{e:limhj1}).
We apply the implication from (2) to (1) in Proposition \ref{equaivalence1}.

Let $a_\infty \in X_{\infty}$, $t_\infty \in (0,\infty)$, and $\phi \in C^{1}((0,\infty))$.
The condition (S1) implies that there exists $a_n \in X_{n}$ such that $f_{n}(a_n) \to a_{\infty}$. 
See Remark \ref{points} for the details.
We assume 
$$
F_{\infty}(t,x):=u^\infty(t,x)-\phi(t)-\frac{k}{2}d_{\infty}(a_\infty,x)^2
$$
has a local strict maximum at $(t_\infty,x_\infty) \in (0,\infty) \times X_{\infty}$. 
Let $l>0$.
We consider a function
$$
F_{n}(t,x):=(u^n)^{\ast}(t,x)-\phi(t)-\frac{k}{2}d_{n}(a_n,x)^2-\frac{l}{2}d_n(x_n,x)^2-\frac{l}{2}|t_n-t|^2.
$$
By the definition of $u^{\infty}$ and the upper semicontinuity of $u^{\infty}$, there exist $x_n \in X_{n}$  and $t_n \in (0,\infty)$ such that $f_n(x_n) \to x_{\infty}$, $t_n \to t_\infty$ and $\lim_{n \to \infty} (u^{n})^{\ast}(t_n, x_{n})=u^{\infty}(t_\infty, x_{\infty})$.

Since $F_n$ is a upper semicontinuous function, there exists a maximizer $(\hat{t}_n, \hat{x}_{n})$ of $F_{n}$ on $[t_n-d,t_n+d] \times \overline{B_{d}(x_{n})}$ with $0<d<\frac{t_n}{2}$.
This implies
\begin{equation}\label{e:stability1}
  \begin{aligned}
    &(u^{n})^{\ast}(t_n,x_{n})
      -\phi(t_n)
      -\frac{k}{2}\,d_{n}(a_{n},x_{n})^{2} \\
    &\le
      (u^{n})^{\ast}(\hat{t}_n,\hat{x}_{n})
      -\phi(\hat{t}_n)
      -\frac{k}{2}\,d_{n}(a_{n},\hat{x}_{n})^{2} 
      -\frac{l}{2}\,d_{n}(x_{n},\hat{x}_{n})^{2}
      -\frac{l}{2}| t_n-\hat{t}_n|^{2}.
  \end{aligned}
\end{equation}

We claim that the left hand side of (\ref{e:stability1}) converges to $u^{\infty}(t_\infty, x_{\infty})-\phi(t_\infty)-\frac{k}{2}d_{\infty}(a_{\infty},x_{\infty})^2$. 
Indeed, by the definition of $f_{n}$, we compute directly
\begin{align}
d_{\infty}\bigl(f_{n}(a_{n}),f_n(x_{n})\bigr)-\epsilon_{n} 
\le d_{n}(a_{n},x_{n}) \notag
&\le d_{\infty}\bigl(f_{n}(a_{n}),f_n(x_{n})\bigr)+\epsilon_n.
\end{align}
Hence, we get $(u^{n})^{\ast}(t_n, x_{n})-\phi(t_n)-\frac{k}{2}d_{n}(a_{n},x_{n})^2 \to u^{\infty}(t_\infty, x_{\infty})-\phi(t_\infty)-\frac{k}{2}d_{\infty}(a_{\infty},x_{\infty})^2$.

To prove $f_n(\hat{x}_n) \to x_\infty$, we show that $f_n(\hat{x}_n)$ converges to some $c_1 \in X_\infty$ as $n \to \infty$.
Indeed, by the triangle inequality and the Cauchy inequality, we have
\begin{equation}\label{e:stability3}
  \begin{aligned}
    &d_{\infty}\bigl(a_\infty,f_{n}(\hat{x}_{n})\bigr)^{2}\\
      &\le
      2\Bigl(d_{\infty}\bigl(a_\infty,f_{n}(a_{n})\bigr)+\epsilon_n\Bigr)^{2}
      +2\Bigl(d_{\infty}\bigl(f_{n}(a_{n}),f_{n}(\hat{x}_{n})\bigr)-\epsilon_n\Bigr)^{2}.
  \end{aligned}
\end{equation}
We evaluate the second term in the right hand side of (\ref{e:stability3}).
By the definition of $f_n$, we have
\[
\left|   d_{n}(a_n,\hat{x}_n) - d_{\infty}\bigl(f_n(a_n), f_n(\hat{x}_n)\bigr) \right| \le \,\epsilon_n,
\]
and
\[
\left|d_{n}(a_n,x_n) - d_{\infty}\bigl(f_n(a_n), f_n(x_n)\bigr) \right| \le \,\epsilon_n.
\]
Thus, by (\ref{e:stability1}), we deduce
\begin{align}\label{e:stability4}
\begin{aligned}
  &\frac{k}{2}\Bigl(d_{\infty}\bigl(f_{n}(a_{n}),f_{n}(\hat{x}_{n})\bigr)-\epsilon_n\Bigr)^2
  -\frac{k}{2}\Bigl(d_{\infty}\bigl(f_{n}(a_n),f_n(x_{n})\bigr)+\epsilon_{n}\Bigr)^2 \\
  &\le \frac{k}{2}\Bigl(d_{\infty}\bigl(f_{n}(a_{n}),f_{n}(\hat{x}_{n})\bigr)-\epsilon_n\Bigr)^2
       -\frac{k}{2}\Bigl(d_{\infty}\bigl(f_{n}(a_n),f_n(x_{n})\bigr)+\epsilon_{n}\Bigr)^2 \\
  &\quad{}+\frac{l}{2}d_n(x_n,\hat{x}_n)^2
           +\frac{l}{2}|t_n-\hat{t}_n|^2 \\[6pt]
  &\le (u^{n})^{\ast}(\hat{t}_n,\hat{x}_{n})-(u^{n})^{\ast}(t_n,x_{n})
       +\phi(\hat{t}_n)-\phi(t_n).
\end{aligned}
\end{align}
By \ref{i:S4}, we have
\begin{equation}\label{e:stability5}
(u^{n})^{\ast}(\hat{t}_n, \hat{x}_{n})-(u^{n})^{\ast}(t_n, x_{n})\le 2M.
\end{equation}
From (\ref{e:stability3}), (\ref{e:stability4}), and (\ref{e:stability5}),
we get
\begin{align*}
&d_{\infty}\bigl(a_\infty,f_n(\hat{x}_n)\bigr)^2\\
&\le 2\Bigl(d_{\infty}\bigl(a_\infty,f_n(a_n)\bigr)+\epsilon_n\Bigr)^2+2\Bigl(d_{\infty}\bigl(f_{n}(a_n),f_n(x_{n})\bigr)+\epsilon_{n}\Bigr)^2+\frac{8M}{k}+\frac{4L_{\phi}d}{k},
\end{align*}
where $L_\phi$ is a Lipschitz constant of $\phi$ on some closed set including $[t_n-d,t_n+d]$. 
Thus, for large $n$, some closed ball independent of $n$ in $X_\infty$ contains $f_n(\hat{x}_n)$.  
Consequently, by the local compactness of $X_\infty$, we have $f_n(\hat{x}_n)$ converges to some $c_1 \in X_\infty$ up to subsequences.

We prove that $f_{n}(\hat{x}_{n}) \to x_{\infty}$ and $\hat{t}_n\to t_\infty$.
Using the triangle inequality and the definition of $f_n$, we have
\begin{align*}
d_{\infty}\bigl(x_\infty,f_{n}(\hat{x}_{n})\bigr)
&\le d_{\infty}\bigl(x_\infty,f_n(x_n)\bigr)+d_{n}(x_n,\hat{x}_{n})+\epsilon_{n}\\
&\le d_{\infty}\bigl(x_\infty,f_n(x_n)\bigr)+d+\epsilon_{n}.
\end{align*}
Letting $n \to \infty$, we obtain $c_1 \in \overline{B_d(x_\infty)}$. Since $\hat{t}_n \in [t_n - d, t_n + d]$, by passing to a subsequence (still denoted by the same index), we deduce that $\hat{t}_n$ converges to some $c_2 \in (0, \infty)$. Furthermore, letting $n \to \infty$ in (\ref{e:stability1}), we obtain that $f_n(\hat{x}_n) \to c_1 = x_\infty$ and \(\hat{t}_n \to c_2 = t_\infty\), independently of the choice of subsequence.

We are now in a position to prove $u^{\infty}$ is a metric viscosity subsolution of (\ref{e:limhj1}).
Note that, by choosing \( l \) sufficiently large, one may ensure that \( F_n \) attains its maximum at an interior point of $[t_n-d,t_n+d] \times \overline{B_{d}(x_{n})}$. 
Indeed, in the middle term of (\ref{e:stability4}), we observe that $\frac{k}{2}\Bigl(d_{\infty}\bigl(f_{n}(a_{n}),f_{n}(\hat{x}_{n})\bigr)-\epsilon_n\Bigr)^2
  -\frac{k}{2}\Bigl(d_{\infty}\bigl(f_{n}(a_n),f_n(x_{n})\bigr)+\epsilon_{n}\Bigr)^2 \to 0$.
Therefore, by choosing \(l\) sufficiently large, we can make both \(d_n(x_n,\hat{x}_n)\) and \(|t_n - \hat{t}_n|\) arbitrarily small, uniformly in all sufficiently large \(n\).

Consider the subsolution test function
$$
\left(\phi_1(t,x), \phi_2(t,x)\right)=\left(\phi(t)+\frac{k}{2}d_{n}(a_n,x)^2,\frac{l}{2}d_n(x_n,x)^2+\frac{l}{2}|t_n-t|^2\right).
$$
Since $u^n$ is a metric viscosity subsolution of (\ref{e:ghj1}), it follows that
$$
\partial_t \phi(\hat{t}_n)+l(\hat{t}_n-t_n)+\left(H_{n}\right)_{ld_n(\hat{x}_n,x_n)}\bigl(\hat{x}_n, kd_{n}(a_n,\hat{x}_n)\bigr)\le 0.
$$
Moreover, by \ref{i:S2}, we have $\left(H_{n}\right)_{ld_n(\hat{x}_n,x_n)}\bigl(\hat{x}_n, kd_{n}(a_n,\hat{x}_n)\bigr)=H_{n}\bigl(\hat{x}_n, kd_{n}(a_n,\hat{x}_n)-ld_n(\hat{x}_n,x_n)\bigr)$.
Consequently, we obtain
\begin{equation}\label{e:stability6}
\partial_t \phi(\hat{t}_n)+l(\hat{t}_n-t_n)+H_{n}\bigl(\hat{x}_n, kd_{n}(a_n,\hat{x}_n)-ld_n(\hat{x}_n,x_n)\bigr)\le 0.
\end{equation}
We deduce from \ref{i:S2}, the definition of $f_n$, and (\ref{e:stability6}) that
\begin{align*}\label{e:stability7}
&\partial_t \phi(\hat{t}_n)
  + l(\hat{t}_n - t_n)\notag\\
&\quad
  + H_n\!\Bigl(
      \hat{x}_n,\,
      k\,d_\infty\bigl(f_n(a_n), f_n(\hat{x}_n)\bigr)
      - l\,d_\infty\bigl(f_n(\hat{x}_n), f_n(x_n)\bigr)-(k+l)\epsilon_n
    \Bigr)
  \le 0.
\end{align*}

Letting $n \to \infty$ and applying \ref{i:S3}, we have
$$
\partial_t \phi(t_\infty)+H_\infty\bigl(x_\infty,kd_\infty(a_\infty,x_\infty)\bigr)\le 0.
$$
Thus, by Proposition \ref{equaivalence1}, $u^\infty$ is a metric viscosity subsolution of (\ref{e:limhj1}).

Finally, we prove that $u^{\infty}$ satisfies the initial condition (\ref{e:limhj2}). Let $x_\infty \in X_\infty$.
Take arbitrary $t_n \in (0,\infty)$ and $x_n \in X_n$ such that 
$t_n \to 0$ and $f_n(x_n) \to x_\infty$ in $X_\infty$.
By the continuity at the initial time of $u^n$, we have
$$
u^\infty(0,x_\infty)=\sup_{(t_n,x_n)}\limsup_{n \to \infty}(u^{n})^{\ast}(t_n,x_n) \le\sup_{(x_n)}\limsup_{n \to \infty}(u^n)^{\ast}(0,x_n).
$$
Since $u^{n}$ satisfies the initial condition (\ref{e:ghj2}), we get
\begin{equation}\label{e:stability8}
(u^n)^{\ast}(0,x_n) \le g^n(x_n).
\end{equation}
By \ref{i:S5} and the definition of the $\epsilon_n$-inverse, the right hand side of (\ref{e:stability8}) can be estimated from above as
\begin{align}\label{e:stability9}
\begin{aligned}
  &g^{n}(x_{n})\le g^{\infty}(x_{\infty})
     +\eta(3\epsilon_n)
     +\Bigl|g^{n} \bigr(f_{n}^{'} \circ f_n(x_{n})\bigr)-g^{\infty}(x_{\infty})\Bigr|.
\end{aligned}
\end{align}
Furthermore, by the triangle inequality, \ref{i:S5} and the definition of the $\epsilon_n$-inverse,
we have
\begin{equation}\label{e:stability10}
\begin{aligned}
&\bigl|g^{n}\bigl(f_{n}'\circ f_n(x_{n})\bigr)-g^{\infty}(x_{\infty})\bigr|\\
  &\le \eta\Bigl(d_\infty\bigl(f_n(x_n),x_\infty\bigr)+4\epsilon_n\Bigr)
 +\bigl|g^{n}\bigl(f_{n}'(x_\infty)\bigr)-g^{\infty}(x_\infty)\bigr|.
\end{aligned}
\end{equation}
Thus, in view of (\ref{e:stability8}), (\ref{e:stability9}), (\ref{e:stability10}) and \ref{i:S5}, we conclude that
\begin{align*}
&u^\infty(0,x_\infty)-g^\infty(x_\infty) \le\sup_{(x_n)}\limsup_{n \to \infty}\left\{(u^n)^{\ast}(0,x_n)-g^\infty(x_\infty)\right\}\\
&\le \sup_{(x_n)}\limsup_{n \to \infty}\Big\{\eta(3\epsilon_n)+\eta\Bigl(d_\infty\bigl(f_n(x_n),x_\infty\bigr)+4\epsilon_n\Bigr)
 +\bigl|g^{n}\bigl(f_{n}'(x_\infty)\bigr)-g^{\infty}(x_\infty)\bigr|\Big\}\\
&\le0
\end{align*}

In the same way, we can prove that $u_\infty$ is a metric viscosity supersolution of (\ref{e:limhj1}), (\ref{e:limhj2}). 
\end{proof}
By Theorem \ref{stability1}, the convergence of metric viscosity solutions follows immediately.
\begin{cor}[Uniform convergence]\label{unifconv}
Let $(X_n, d_n)$ and $(X_\infty, d_\infty)$ be compact geodesic spaces.
Let $u^n$ be a continuous metric viscosity solution of (\ref{e:ghj1}), (\ref{e:ghj2}).
We make the same assumptions as in Theorem \ref{stability1}.
We also assume that the comparison principle holds for the Hamilton--Jacobi equation (\ref{e:limhj1}), (\ref{e:limhj2}); that is, for continuous metric viscosity subsolution $u$ and supersolution $v$ of (\ref{e:limhj1}), (\ref{e:limhj2}), we have $u \le v$ on $[0,\infty) \times X_\infty$. 
 Then, $u^\infty =u_\infty(=:\overline{u})$ on $[0,\infty) \times X_\infty$ and $u^n$ converges uniformly to $\overline{u}$:
$$
\lim_{n \to \infty}u^n(t_n,x_n)=\overline{u}(t,x),\quad (t,x) \in [0,\infty) \times X_\infty,
$$
where $x_n \in X_n$ satisfies $f_n(x_n) \to x$ in $X_\infty$ and $t_n \in [0,\infty)$ satisfies $t_n \to t$ in $[0,\infty)$ as $n \to \infty$.
\end{cor}
\begin{proof}
By the definition of $u^\infty$ and $u_\infty$, we have $u^\infty \ge u_\infty$ in $[0,\infty) \times X_\infty$. Furthermore, by Theorem \ref{stability1} and the comparison principle, we have $u^\infty \le u_\infty$. Consequently, $u^{\infty}=u_\infty$ in $[0,\infty) \times X_\infty$.

By a standard argument in the theory of viscosity solutions, we obtain
\[
  \lim_{n \to \infty} u^n(t_n, x_n) \;=\; \overline{u}(t, x),
  \quad (t,x) \in [0,\infty) \times X_\infty,
\]
where $x_n \in X_n$ satisfies $f_n(x_n) \to x$ and $t_n \in [0,\infty)$ satisfies $t_n \to t$ as $n \to \infty$.
Thus, we obtain the conclusion.
\end{proof}
\section{Stability of the dual Kantorovich problem}
To prove Proposition \ref{stakan}, we demonstrate the convergence of the Hopf--Lax semigroups under Gromov--Hausdorff convergence. 
\begin{lem}\label{important}
Let $(X_n, d_n)$ and $(X_\infty, d_\infty)$ be compact geodesic spaces.
Assume \ref{i:S1}.
We also assume that $g^n: X_n \to \mathbb{R}$ and $g^\infty: X_\infty \to \mathbb{R}$ satisfy the following condition. 
\begin{enumerate}
\renewcommand{\labelenumi}{(S5)$'$}
\renewcommand{\theenumi}{(S5)$'$}
\item 
\label{i:S5pr}
The function $g^n$ is equi-Lipschitz in $n$, that is, there exists a positive constant $L$ such that, for $x,y \in X_n$, the inequality 
    $$
    |g^n(x)-g^n(y)|\le L d_n(x,y)
    $$ holds.
Furthermore, for $x \in X_\infty$, $g^n(f^{'}_{n}(x))$ converges to $g^\infty(x)$ as $n \to \infty$. 
\end{enumerate}
Then, for $t>0$, the Hopf--Lax semigroup $Q_t g^n$
converges uniformly to $Q_t g^\infty$ as $n \to \infty$. More precisely, for $t>0$, the convergence
\begin{equation}\label{uconv}
\lim_{n \to \infty}Q_t g^n(x_n)=Q_t g^{\infty}(x),\quad x \in X_\infty
\end{equation}
holds, where $x_n \in X_n$ satisfies $f_n(x_n) \to x$ in $X_{\infty}$ as $n \to \infty$.
\end{lem}
\begin{proof}
We note that $Q_t g^n$ is the unique metric viscosity solution of 
$$
\partial_{t}{u^n}(t,x)+\frac{1}{2}|\nabla u^n(t,x)|^{2}=0,\quad (t,x) \in (0,\infty) \times X_n,
$$
$$
u^{n}(0,x)=g^n(x), \quad x \in X_n.
$$
Likewise, $Q_t g^\infty$ is also the unique metric viscosity solution on $X_\infty$; see \cite[Theorem 7.7]{GS14}.

We verify in turn the assumptions required for Corollary~\ref{unifconv}.
Conditions \ref{i:S2}, \ref{i:S3} can be checked directly. 
We can check that $u^n$ is equi-Lipschitz in $n$ by the equi-coercivity of the Hamiltonian and the equi-Lipschitz continuity of $g^n$ in $n$ (which follows by a standard argument from viscosity solution theory).
Thus,by the compactness of $X_n$, we can prove that $u^n$ satisfies \ref{i:S4}.
The equi-Lipschitz continuity of \(u^n\) in $n$ also implies the continuity at initial time of $u^n$ and the semicontinuity of both \( u^\infty \) and \( u_\infty \).
Moreover, \ref{i:S5} follows immediately from \ref{i:S5pr}.
Thus, by Corollary \ref{unifconv}, we obtain (\ref{uconv}).
\end{proof}
\begin{rem}
By the definition of an $\epsilon_n$-inverse $f_n^{'}$ of $f_n$, the convergence (\ref{uconv}) corresponds to the standard notion of convergence in optimal transport theory \cite{V09}. More precisely, the convergence 
\begin{equation*}
\lim_{n \to \infty}Q_t g^n \circ f_n^{'}(x)=Q_t g^{\infty}(x)
\end{equation*}
holds for $t>0$ and uniformly for $x \in X_\infty$.
\end{rem}
We are now in position to prove Proposition \ref{stakan} using Lemma \ref{important}. 
To this end, we may assume that $f_n^{'}$ is a Borel map.
\begin{proof}[Proof of Proposition \ref{stakan}]
Let $\phi^n$ be a maximizer of (\ref{kann}).
Without loss of generality, we can assume that $\phi^n$ is $\frac{d_n^2}{2}$-convex \cite[Theorem 5.10 (iii)]{V09}. 
Fix a point $z \in X_{\infty}$. 
Define $c^n:=\phi^n\bigl(f^{'}_{n}(z)\bigr)$. 
Then, $\bigl(\phi^n-c^n\bigr)\bigl(f^{'}_{n}(z)\bigr)=0$ holds. 
Since $\phi^n$ is $\frac{d_n^2}{2}$-convex, $\phi^n-c^n$ is $1$-Lipschitz with respect to $\frac{d^{2}_{n}}{2}$. 
Thus, we have
\begin{equation}\label{e:pfkan1}
\begin{aligned}
&\bigl|(\phi^{n}-c^{n})(x)-(\phi^{n}-c^{n})(y)\bigr|\\
&\le \frac{d_{n}(x,y)^{2}}{2}\le \frac{\operatorname{diam}(X_{n})}{2}\, d_{n}(x,y),
\quad x,y \in X_{n}.
\end{aligned}
\end{equation}
where $\operatorname{diam} (X_n):=\sup_{x,y \in X_n}d_n(x,y)$. 
By the compactness of $X_\infty$ and the definition of $f_n$, we get
\begin{align}\label{e:pfkan2}
\operatorname{diam} (X_n)=\sup_{x,y \in X_n}d_n(x,y)\le \sup_{x,y \in X_n}d_\infty\bigl(f_n(x),f_n(y)\bigr)+\epsilon_n\le K
\end{align}
for some positive constant $K$ independent of $n$.
Combining (\ref{e:pfkan1}) and (\ref{e:pfkan2}), $\phi^n-c^n$ is equi-Lipschitz with respect to $d_n$.
Furthermore, by the choice of $z$, we have
$$
|(\phi^n-c^n)(x)|\le \frac{K}{2} d_n(x,z)\le \frac{K^2}{2},\quad x \in X_n.
$$
Thus, $\phi^n-c^n$ is uniformly bounded. 
By the Ascoli--Arzelà Theorem, $(\phi^n-c^n)\circ f^{'}_{n}$ converges uniformly on $X_{\infty}$ to some Lipschitz function $\overline{\phi}$ up to subsequences.
Applying Lemma \ref{important}, we obtain the uniform convergence of 
$Q_{1}(\phi^n - c^n) \circ f^{'}_{n}$ to $Q_{1}(\overline{\phi})$.

Since $\phi^n$ is a maximizer of (\ref{kann}), we have
\begin{equation}\label{e:pfkan3}
\aligned
&\int_{X_n} Q_{1}\bigl(\phi^{n}-c^{n}\bigr)\,d\nu_{n}
      -\int_{X_n} \bigl(\phi^{n}-c^{n}\bigr)\,d\mu_{n}\\
  &=\int_{X_n} Q_{1}(\phi^{n})\,d\nu_{n}
      -\int_{X_n} \phi^{n}\,d\mu_{n}
   =\frac12\,W_{2}(\mu_{n},\nu_{n})^{2}.
\endaligned
\end{equation}

By the triangle inequality, we observe
\begin{equation}\label{e:pfkan4}
\begin{aligned}
\Bigl|
  &\int_{X_n} Q_{1}\bigl(\phi^{n}-c^{n}\bigr)(x)\,d\nu_{n}
  -\int_{X_{\infty}} Q_{1}(\overline{\phi})(x)\,d\nu_{\infty}
\Bigr|\\
&\le
  \int_{X_n}
    \Bigl|
      Q_{1}\bigl(\phi^{n}-c^{n}\bigr)(x)
      -Q_{1}\bigl(\phi^{n}-c^{n}\bigr)\!\bigl(f^{'}_{n}\circ f_{n}(x)\bigr)
    \Bigr|
  \,d\nu_{n} \\
&\quad
  +\Bigl|
    \int_{X_{\infty}}
      Q_{1}\bigl(\phi^{n}-c^{n}\bigr)\!\bigl(f^{'}_{n}(x)\bigr)
      \,d \bigl(f_{n}\bigr)_{\#}\nu_{n}
    -\int_{X_{\infty}}
      Q_{1}\bigl(\phi^{n}-c^{n}\bigr)\!\bigl(f^{'}_{n}(x)\bigr)
      \,d\nu_{\infty}
  \Bigr| \\
&\quad
  +\Bigl|
    \int_{X_{\infty}}
      Q_{1}(\overline{\phi})(x)\,d\nu_{\infty}
    -\int_{X_{\infty}}
      Q_{1}\bigl(\phi^{n}-c^{n}\bigr)\!\bigl(f^{'}_{n}(x)\bigr)
      \,d\nu_{\infty}
  \Bigr|.
\end{aligned}
\end{equation}
Since $Q_{1}(\phi^{n}-c^n)$ is equi-Lipschitz, the first term in the right hand side of (\ref{e:pfkan4}) converges to $0$ as $n \to \infty$.
The second term is estimated above as
\begin{equation*}
\begin{aligned}
&\left|
  \int_{X_\infty}
      Q_{1}\bigl(\phi^{n}-c^{n}\bigr)\bigl(f^{'}_{n}(x)\bigr)\,
      d(f_n)_{\#}\nu_{n}
  - \int_{X_\infty}
      Q_{1}(\overline{\phi})(x)\,
      d(f_n)_{\#}\nu_{n}
\right| \\[4pt]
+{}&\left|
  \int_{X_\infty}
      Q_{1}(\overline{\phi})(x)\,
      d(f_n)_{\#}\nu_{n}
  - \int_{X_\infty}
      Q_{1}\bigl(\phi^{n}-c^{n}\bigr)\bigl(f^{'}_{n}(x)\bigr)\,
      d\nu_{\infty}
\right|.
\end{aligned}
\end{equation*}

Thus, the second term converges to $0$ by the uniform convergence of $Q_{1}(\phi^n - c^n) \circ f^{'}_{n}$.
The third term also tends to zero, again due to the uniform convergence of $Q_{1}(\phi^n - c^n) \circ f^{'}_{n}$.
Consequently, the right side of (\ref{e:pfkan4}) vanishes as $n \to \infty$.

A similar argument applies to the term $\int_{X_{n}}(\phi^{n}-c^n) d\mu_{n}$ in (\ref{e:pfkan3}). Then, by the stability of $2$-Wasserstein distance under measured Gromov--Hausdorff convergence \cite[Theorem 28.6]{V09}, we conclude
$$
\int_{X_{\infty}} Q_{1}(\overline{\phi})\,d\nu_{\infty}-\int_{X_{\infty}}\overline{\phi} \,d\mu_{\infty}=\frac{1}{2}W_2(\mu_\infty,\nu_\infty)^2.
$$
Therefore, $\overline{\phi}$ is a maximizer of (\ref{kani}).
\end{proof}
\section{Acknowledgments}
The author is grateful to Nao Hamamuki and Atsushi Nakayasu for fruitful discussions and insightful suggestions during the preparation of this paper. 
This work was supported by JST SPRING, Grant Number JPMJSP2119.

\end{document}